\documentclass{article}

\usepackage{arxiv}
\usepackage[applemac]{inputenc} 		% allow utf-8 input
\usepackage[T1]{fontenc}    		% use 8-bit T1 fonts
\usepackage[colorlinks]{hyperref}      % hyperlinks
\usepackage{url}            			% simple URL typesetting
\usepackage{amsthm}
\usepackage{algorithm}
\usepackage{algorithmicx}
\usepackage{algpseudocode}
\usepackage{amsfonts}       		% blackboard math symbols
\usepackage{amsmath}
\usepackage{array}
\usepackage{amssymb}

\usepackage{booktabs}       		% professional-quality tables
\usepackage{caption}
\usepackage{comment}
\usepackage{dsfont}
\usepackage{enumitem}
\usepackage{framed}
\usepackage{fancyhdr}
\usepackage{graphicx}
\usepackage{indentfirst,latexsym,bm}
\usepackage{lipsum}				% Can be removed after putting your text content
\usepackage{microtype}      		% microtypography
\usepackage{mathtools}
\usepackage[square,sort,comma,numbers]{natbib}
\usepackage{nicefrac}       		% compact symbols for 1/2, etc.
\usepackage{subcaption}
\usepackage{wrapfig}
\usepackage{epstopdf}
\usepackage{xcolor}         % colors
\usepackage{tikz}
\usepackage{pgfplots}
\usepackage{hyperref}[6.83]
\usetikzlibrary{arrows.meta,positioning}
\pgfplotsset{compat=newest}
\usepgfplotslibrary{fillbetween}
\usetikzlibrary{shapes,decorations}
\usetikzlibrary{fit}

\colorlet{color1}{blue}
\colorlet{color2}{red!50!black}

\definecolor{ivory}{RGB}{218,215,203}

\definecolor{cuhkp}{RGB}{98,56,105} 	% purple dark
\definecolor{cuhkpl}{RGB}{152,24,147} 	% purple light
\definecolor{cuhkb}{RGB}{219,160,1} 	% ocher
\definecolor{cuhkbd}{RGB}{178,129,0} 	% ocher dark
\definecolor{cuhkr}{RGB}{88,35,155}  	% magenta-red
\definecolor{darkblue}{rgb}{0,0.08,0.45} %dark blue

\definecolor{bluep}{RGB}{0,128,255}
\newcommand{\revise}[1]{{\color{black}{#1}}}

\hypersetup{
	colorlinks=true,
	frenchlinks=false,
	pdfborder={0 0 0},
	naturalnames=false,
	hypertexnames=false,
	breaklinks,
	linkcolor=darkblue,
	citecolor=darkblue,
	filecolor=darkblue,
	urlcolor = black,
}

\RequirePackage[capitalize,nameinlink]{cleveref}[0.19]

\crefname{section}{section}{sections}
\crefname{subsection}{subsection}{subsections}

\Crefname{figure}{Figure}{Figures}

\crefformat{equation}{\textup{#2(#1)#3}}
\crefrangeformat{equation}{\textup{#3(#1)#4--#5(#2)#6}}
\crefmultiformat{equation}{\textup{#2(#1)#3}}{ and \textup{#2(#1)#3}}
{, \textup{#2(#1)#3}}{, and \textup{#2(#1)#3}}
\crefrangemultiformat{equation}{\textup{#3(#1)#4--#5(#2)#6}}%
{ and \textup{#3(#1)#4--#5(#2)#6}}{, \textup{#3(#1)#4--#5(#2)#6}}{, and \textup{#3(#1)#4--#5(#2)#6}}

\Crefformat{equation}{#2Equation~\textup{(#1)}#3}
\Crefrangeformat{equation}{Equations~\textup{#3(#1)#4--#5(#2)#6}}
\Crefmultiformat{equation}{Equations~\textup{#2(#1)#3}}{ and \textup{#2(#1)#3}}
{, \textup{#2(#1)#3}}{, and \textup{#2(#1)#3}}
\Crefrangemultiformat{equation}{Equations~\textup{#3(#1)#4--#5(#2)#6}}%
{ and \textup{#3(#1)#4--#5(#2)#6}}{, \textup{#3(#1)#4--#5(#2)#6}}{, and \textup{#3(#1)#4--#5(#2)#6}}

\crefdefaultlabelformat{#2\textup{#1}#3}

\theoremstyle{plain}
\newtheorem{theorem}{Theorem}
\newtheorem{lemma}{Lemma}
\newtheorem{corollary}{Corollary}

\newtheorem{definition}{Definition}

\newtheorem{fact}{Fact}

\theoremstyle{definition}

\graphicspath{{./figs/}}

\DeclareMathOperator*{\argmin}{argmin}

\newcommand{\R}{\mathbb{R}}

\newcommand{\Exp}{\mathbb{E}}

\newcommand{\dist}{\operatorname{dist}}

\newcommand{\proj}{\operatorname{proj}}

\newcommand{\sL}{{\sf L}}

\newcommand{\SubGrad}{\mathsf{SubGrad}}
\newcommand{\ProxSubGrad}{\mathsf{Prox}\text{-}\mathsf{SubGrad}} 
\newcommand{\TSM}{\mathsf{TSM}} 
 
\newcommand{\SSM}{\mathsf{SSM}}

\newcommand{\be}{\begin{equation}}
\newcommand{\ee}{\end{equation}}

\newcommand\prox{{\operatorname{prox}}}

\newcommand{\calA}{\mathcal{A}}
\newcommand{\calB}{\mathcal{B}}
\newcommand{\calC}{\mathcal{C}}
\newcommand{\calD}{\mathcal{D}}

\newcommand{\calH}{\mathcal{H}}

\newcommand{\calO}{\mathcal{O}}

\newcommand{\calS}{\mathcal{S}}

\newcommand{\calU}{\mathcal{U}}
\newcommand{\calV}{\mathcal{V}}

\newcommand{\calX}{\mathcal{X}}

\newcommand{\setR}{\mathbb{R}}

\newcommand{\en}{\begin{equation*}}
\newcommand{\een}{\end{equation*}}

\title{Revisiting Subgradient Method:  Complexity and Convergence Beyond Lipschitz Continuity\thanks{Xiao Li  was supported in part by the National Natural Science Foundation of China (NSFC) under Grant No. 12201534 and in part by the Shenzhen Science and Technology Program under Grant No. RCYX20221008093033010. Lei Zhao and Daoli Zhu were partially supported by the National Key R\&D Program of China (Grant No. 2023YFA0915202); the Major Project of the National Natural Science Foundation of China (Grant No.72293582); the Fundamental Research Funds for the Central Universities (the Interdisciplinary Program of Shanghai Jiao Tong University) (Grant No. YG2024QNA36). Lei Zhao was also partially supported by the Startup Fund for Young Faculty at SJTU (SFYF at SJTU) (Grant No. 22X010503839). Anthony Man-Cho So was supported in part by the Hong Kong Research Grants Council (RGC) General Research Fund (GRF) project CUHK 14204823.}}

\author{
	Xiao Li\\
	School of Data Science\\
	The Chinese University of Hong Kong, Shenzhen\\
	Shenzhen, China\\
	\texttt{lixiao@cuhk.edu.cn} 
	\And
    Lei Zhao\\
    Institute of Translational Medicine \\
    National Center for Translational Medicine\\
    Shanghai Jiao Tong University\\
    Shanghai, China\\
    \texttt{l.zhao@sjtu.edu.cn} 
    \And
    Daoli Zhu\\
    Antai College of Economics and Management \\
    Shanghai Jiao Tong University\\
    Shanghai, China\\
    School of Data Science\\
    The Chinese University of Hong Kong, Shenzhen\\
    Shenzhen, China\\
    \texttt{dlzhu@sjtu.edu.cn}
    \And
    Anthony Man-Cho So \\
     Department of Systems Engineering and Engineering Management \\
    The Chinese University of Hong Kong \\ 
    \texttt{manchoso@se.cuhk.edu.hk} 
}

\begin{document}

\maketitle

\begin{abstract}
	The subgradient method is one of the most fundamental algorithmic schemes for nonsmooth optimization. The existing complexity and convergence results for this method are mainly derived for Lipschitz continuous objective functions. In this work, we first extend the typical iteration complexity results for the subgradient method to cover non-Lipschitz convex and weakly convex minimization. Specifically, for the convex case, we can drive the suboptimality gap to below $\varepsilon$ in $\mathcal{O}( \varepsilon^{-2} )$ iterations; for the weakly convex case, we can drive the gradient norm of the Moreau envelope of the objective function to below $\varepsilon$ in $\mathcal{O}( \varepsilon^{-4} )$ iterations. Then, we provide convergence results for the subgradient method in the non-Lipschitz setting when proper diminishing rules on the step size are used. In particular, when $f$ is convex, we establish an $\calO(\log(k)/\sqrt{k})$ rate of convergence in terms of the suboptimality gap, where $k$ represents the iteration count. With an additional quadratic growth property, the rate is improved to $\calO(1/k)$ in terms of the squared distance to the optimal solution set.  When $f$ is weakly convex, asymptotic convergence is established. Our results neither require any modification to the subgradient method nor impose any growth condition on the subgradients, while our analysis is surprisingly simple. The central idea to deriving the aforementioned results is that the dynamics of the step size rule fully controls the movement of the subgradient method, which leads to the boundedness of the iterates and allows a trajectory-based analysis to be conducted to establish the desired results. To further illustrate the wide applicability of our framework, we extend the aforementioned iteration complexity results to cover the truncated subgradient, the stochastic subgradient,  and the proximal subgradient methods for non-Lipschitz convex / weakly convex objective functions.
\end{abstract}

\section{Introduction}\label{sec:intro}

Let $f:\mathbb{R}^d \rightarrow \mathbb{R}$ be a convex / weakly convex function (i.e., $f(\cdot) + \tfrac{\rho}{2} \| \cdot \|^2$ is convex for some $\rho \ge 0$), which can be nonsmooth. Consider the following problem: 
\begin{equation}\label{eq:problem}
	\min_{x\in \R^d} \ f(x).
\end{equation}
Throughout this work, we use $\calX^*$ to denote the set of optimal solutions to problem \eqref{eq:problem}. We assume that $\calX^*$ is non-empty and use $f^*$ to denote the optimal value of problem \eqref{eq:problem}. Also, we use $\partial f(x)$ and $g(x) \in \partial f(x)$ to denote the subdifferential and a subgradient of $f$ at $x$, respectively. We defer the precise definition of the subdifferential to \Cref{sec:prelimilary}.  In this work, we are interested in applying the subgradient method (hereafter, $\SubGrad$) for solving problem \eqref{eq:problem}.  $\SubGrad$ is one of the most fundamental optimization schemes for addressing nonsmooth optimization problems. It proceeds by picking an arbitrary subgradient $g(x^k)\in \partial f(x^k)$ in the $k$-th iteration and searching along the negative direction of the selected subgradient for updating, namely,
\[
	x^{k+1} =x^{k} - \alpha_{k}g(x^k), 
\]
where $\alpha_{k}>0$ is the step size.

\paragraph{The context and our goal.} The iteration complexity and convergence properties of $\SubGrad$ have been extensively studied in the setting where the objective function $f$ is Lipschitz continuous. In the convex case, the standard iteration complexity bound for $\SubGrad$ is given by 
\be\label{eq:Lip cvx complexity}
  f(\tilde x^T) - f^* \leq \calO\left(\frac{1}{\sqrt{T}}\right),
\ee
where the big-O hides constants like the Lipschitz constant of $f$ and the initial suboptimality gap $f(x^0) - f^*$, $T$ is the pre-determined total number of iterations, and $\tilde x^T$ is certain weighted average of the first $T$ iterates. In the weakly convex case, we have the following result for $\SubGrad$ \cite{davis2019stochastic}:
\be\label{eq:Lip wcvx complexity}
   \min_{k = 0,\ldots, T} \ \|\nabla f_\lambda(x^k)\|^2 \leq  \calO\left(\frac{1}{\sqrt{T}}\right).
\ee
Here, $f_\lambda$ is the Moreau envelope of $f$ with parameter $\lambda >0$ (see \Cref{sec:prelimilary} for details). As it turns out, the gradient of $f_\lambda$ can serve as a surrogate stationarity measure for $f$ \cite{davis2019stochastic}. 

It should be noted that the Lipschitz continuity assumption on the objective function $f$ is crucial to the derivations of the complexity results \eqref{eq:Lip cvx complexity} and \eqref{eq:Lip wcvx complexity}. Indeed, such an assumption, together with some diminishing step size rule or a small constant step size,  is used to bound the error term caused by the non-descent nature of $\SubGrad$. However, many applications give rise to objective functions that are not Lipschitz continuous. For example, the SVM \cite{cortes1995support} amounts to solving the problem
\be\label{eq:svm}
\min_{x\in\R^d}\ \left\{f(x) = \frac{1}{n}\sum_{i=1}^n\max\left\{0,1-b_i(a_i^{\top}x)\right\}+\frac{\kappa}{2}\|x\|^2\right\},
\ee
where $\kappa>0$, $a_i\in\R^d$, and $b_i\in\{\pm1\}$ for $i=1,...,n$ are given. It can be easily verified that the objective function $f$ in \eqref{eq:svm}, though convex, is not Lipschitz continuous on $\R^d$.  Indeed, most $\ell_2$-regularized convex learning problems do not satisfy the Lipschitz continuity assumption. 
Another example is the robust matrix sensing problem \cite{li2020nonconvex}
\be
\min_{X\in \R^{d\times r}} \ \left\{ f(X) = \frac{1}{n}\|y - \calA(XX^\top) \|_1\right\},
\label{eq:rms}
\ee
where $y\in \R^n$ is a measurement vector and $\calA:\R^{d\times d} \rightarrow \R^n$ is a linear measurement operator consisting of a set of sensing matrices $A_1,\ldots,A_n \in \R^{d\times d}$. The matrix sensing problem \eqref{eq:rms}, whose objective function $f$ can be verified to be weakly convex but not Lipschitz continuous, encapsulates the (real-valued) robust phase retrieval problem, the robust quadratic sensing problem, and certain robust PCA variants. 

Motivated by the above discussion, our goal in this work is to establish the iteration complexity and convergence of $\SubGrad$ for both convex and weakly convex minimization \emph{without the Lipschitz continuity assumption}. 

%
%\begin{algorithm}[t]
%	\caption{$\SubGrad$: Subgradient method for minimizing $f$ over $\R^d$}
%	{\bf Input:}  Initial point $x^0$.
%	\begin{algorithmic}[1]
%		\For{$k=0,1,\ldots$}
%		\State Compute a subgradient $g(x^k)\in\partial f(x^k)$.
%		\State Update the step size $\alpha_{k}$ according to a certain rule.
%		\State Update  $x^{k+1} =x^{k} - \alpha_{k}g(x^k)$.
%		\EndFor
%	\end{algorithmic}
%	\label{alg:RCS}
%\end{algorithm}

\subsection{New Iteration Complexity Bounds for the Subgradient Method}

One important element in our development is the step size rule, which in the $k$-th iteration takes the form
\[
     \alpha_k = \frac{\beta_k}{\|g(x^k)\|}
\]
for some $\beta_k>0$, where $g(x^k)$ is a subgradient of $f$ at $x^k$. Such a rule is ubiquitous in both the implementation and theoretical analysis of $\SubGrad$; see, e.g.,  \cite{Shor85,CohenZ,Pol87,renegar2016,nesterov2018lectures,grimmer2018,grimmer2019}. The idea behind this normalization is the fact that the norm of the subgradient, i.e., $\|g(x^k)\|$, is not very informative in nonsmooth optimization.  In the case where $\| g(x^k) \| = 0$, we know that $x^k$ is a stationary point of $f$. Thus, the above step size rule is well defined. 

We first present our complexity result for the convex case. The following theorem generalizes the existing complexity result \eqref{eq:Lip cvx complexity}.
\begin{theorem}[complexity of $\SubGrad$ for convex case]\label{theo:complexity_cvx} Suppose that $f$ in \eqref{eq:problem} is convex and the step sizes  $\{\alpha_k\}_{k= 0}^T$ satisfy 
	\be\label{eq:constant step size}
     	\alpha_k = \frac{\beta_k}{\|g(x^k)\|}  \quad \text{with} \quad \beta_k= \frac{c}{\sqrt{T+1}} \quad \text{for} \ k= 0,1,\ldots, T,
	\ee
	where $c>0$ is some constant and $T$ is the pre-determined total number of iterations. 
	Let $\tilde{x}^T=\frac{1}{T+1}\sum_{k=0}^T x^k$ be the averaged iterate and $x^*\in \calX^*$ be some fixed optimal solution to problem \eqref{eq:problem}. Then, the trajectory $\{x^k\}_{k=0}^T$ lies in the ball $ \calA := \{x\in\R^d: \|x-x^*\|^2 \leq \|x^{0}-x^*\|^2+c^2\}$. Consequently, for any bounded open convex set $\mathcal{A}'$ that contains $\calA$, we have
	\[
	f(\tilde{x}^{T})-f^*\leq \frac{\sL_{\mathcal {A}'}\left(\|x^0-x^*\|^2/c + c\right)}{2\sqrt{T+1}},
	\]
	where ${\sf L}_{\mathcal {A}'}>0$ is the Lipschitz constant\footnote{Recall that all finite-valued convex and weakly convex functions are locally Lipschitz; see, e.g., \cite[Proposition 4.4]{Vial83}.\label{foot:local lip}} of $f$ on $\mathcal{A}'$.
\end{theorem}

The $\calO(1/\sqrt{T})$ iteration complexity established in \Cref{theo:complexity_cvx}, which does not require the Lipschitz continuity assumption, matches the one that requires such an assumption. If $f$ is Lipschitz continuous with parameter $\sL$, then we can simply choose $\sL_{\mathcal A'}= \sL$. This recovers the standard complexity result \eqref{eq:Lip cvx complexity} of $\SubGrad$ for minimizing convex Lipschitz functions. We remark that the result in \Cref{theo:complexity_cvx} is not entirely new. It establishes a complexity bound similar to those presented in \cite[Theorem 3.2.2]{nesterov2018lectures} and \cite[Section 1.1]{grimmer2019} in the non-Lipschitz convex setting. Nonetheless, it is important to highlight that our approach to proving this result is different from and simpler than those in \cite{nesterov2018lectures,grimmer2019}. This enables us to derive new result for the weakly convex case, which we now explain. We refer the reader to \Cref{sec:related works} for further discussion.

Unlike in the convex case where we do not need any additional assumption on the objective function, in the weakly convex case we assume that the objective function has \emph{bounded sublevel sets}, i.e., the set $\calS_{\leq s}: = \{x\in \R^d: f(x) \leq s\}$ is bounded for any $s\in \R$. We note that the bounded sublevel sets assumption is satisfied by many weakly convex minimization problems such as the robust matrix sensing problem \eqref{eq:rms} and its variants, as their objective functions are coercive.  Therefore, it can be regarded as a mild assumption. The following is our new iteration complexity bound for $\SubGrad$ when utilized to minimize weakly convex functions. It generalizes the existing complexity result \eqref{eq:Lip wcvx complexity} by removing the Lipschitz continuity assumption on $f$.
  
 \begin{theorem}[complexity of $\SubGrad$ for weakly convex case]\label{theo:complexity wcvx} Suppose that $f$ in \eqref{eq:problem} is $\rho$-weakly convex and has bounded sublevel sets. Suppose further that the step sizes $\{\alpha_k\}_{k= 0}^T$ satisfy the conditions in \eqref{eq:constant step size}. Then, the trajectory $\{x^k\}_{k=0}^T$ lies in the bounded sublevel set $\calB := \{x\in \R^d: f_\lambda (x) \leq f_{\lambda}(x^0) + c^2/2\lambda\}$. Consequently, for any bounded open convex set $\mathcal{B}'$ that contains $\mathcal{B}$, we have 
 	\be\label{eq:wcvx complexity}
 	\min_{0\leq k\leq T}\|\nabla f_{\lambda}(x^k)\|^2 \leq \frac{\sL_{\mathcal{B}'}\left(f_{\lambda}(x^0)-f^*+ c^2 / 2\lambda \right)}{c(1-\lambda\rho)\sqrt{T+1}},
 	\ee
 	where  $\sL_{\mathcal{B}'} > 0$ is the Lipschitz constant of $f$ on $\calB'$. 
 \end{theorem}
%{\color{red}The idea behind the proof of \Cref{theo:complexity wcvx} is similar to that of \Cref{theo:complexity_cvx}. The step size rule enables us to show that the sequence of Moreau envelope function values is uniformly bounded within $T$ iterations. Then, the bounded sublevel sets assumption implies that the sequence $\{x^k\}_{k=0}^T$ lies in the bounded set $\calB$. Then, conducting the trajectory-based analysis yields the result.}

The proofs of \Cref{theo:complexity_cvx} and \Cref{theo:complexity wcvx} can be found in \Cref{sec:SubGrad}.

\paragraph{Convergence results using diminishing step size rules.}  Our analysis is not tailored to the constant choice of $\beta_k$ in \eqref{eq:constant step size}. Actually, we can also utilize diminishing step size rules satisfying $\sum_{k=0}^{\infty} \beta_k^2 <\infty$. One immediate benefit is that we can now show the boundedness of $\{x^k\}$ for all $ k\geq 0$ rather than within the pre-determined $T$ iterations. Therefore, we are able to establish convergence results by properly diminishing $\beta_k$ for non-Lipschitz convex and weakly convex objective functions. To be more specific, for the convex case, we establish an $\calO(\log(k)/\sqrt{k})$ convergence rate in terms of the suboptimality gap  and an $\calO(1/k)$ convergence rate in terms of the squared distance to the optimal solution set if the objective function possesses an additional quadratic growth condition. For the weakly convex case, we prove the asymptotic convergence result $\lim_{k\rightarrow\infty}\|\nabla f_\lambda(x^k)\| = 0$. These convergence results are discussed in detail in \Cref{sec:convergence}. We note that in the weakly convex case, we do not have a convergence rate result. This is expected as no convergence rate result exists even in the Lipschitz continuous setting. We leave the convergence rate analysis of $\SubGrad$ in the weakly convex case as a future work.

\paragraph{Our idea and extensions.}  The conventional analysis of $\SubGrad$ typically bounds the norm of the subgradient used in each iteration by the Lipschitz constant of the objective function $f$. By contrast, the main challenge in our analysis is the absence of a \emph{global} bound on the subgradient norms, rendering the existing analysis inapplicable. To address this issue, we first observe that the step size normalization scheme allows us to derive an initial recursion of $\SubGrad$ whose ascent / error term is independent of the subgradient. Then, by  combining the dynamics of $\beta_k$ with the convexity or weak convexity of $f$, we can establish the boundedness of the sequence $\{x^k\}$. This reveals that a  proper step size rule completely controls the movement of $\SubGrad$ and eliminates the pathological cases where $\{x^k\}$ tends to be unstable. To the best of our knowledge, while a boundedness result has previously been derived for the convex case (see, e.g., \cite{alber1998projected}), we provide a new analogous result for the weakly convex case. Based on this boundedness result, we are able to establish a standard recursion for the trajectory generated by $\SubGrad$, which finally leads to the aforementioned complexity and convergence results. Specifically, consider any bounded open convex set $\calU$ that contains the entire trajectory $\{x^k\}$. For a convex / weakly convex objective function $f$, which is automatically Lipschitz continuous over the bounded set $\calU$ with some parameter $\sL_{\mathcal{U}}>0$ (again, see \Cref{foot:local lip}), we can derive the following upper bound for the subgradients $\{g(x^k)\}$ used in the trajectory $\{x^k\}$:
\be\label{eq:subgrad bound}
    \|g(x^k)\|\leq \sL_{\mathcal{U}}, \quad \forall k\geq 0.
\ee
Plugging this important upper bound into the aforementioned initial recursion yields the standard recursion for the iterates of $\SubGrad$. Finally, standard arguments lead to our iteration complexity and convergence results. The bound \eqref{eq:subgrad bound} is based on the following classic result in variational analysis:
\begin{fact}[Theorem 9.13 of \cite{Rockafellar}] \label{fact}
	For any lower semicontinuous function $h$, the Lipschitz continuity parameter $\sL_\calV$ of $h$ over an open convex set $\calV$ coincides with $\sup \{\|g(x)\|: x\in \calV, g(x) \in \partial h(x)\}$.
\end{fact}
To the best of our knowledge, our work provides the \emph{first} iteration complexity and convergence analysis of $\SubGrad$ for weakly convex minimization without assuming the Lipschitz continuity of the objective function $f$. For the non-Lipschitz convex case, while there are existing iteration complexity bounds (see \Cref{theo:complexity_cvx} and the discussion following it), our technique is more versatile, enabling the analysis of different methods. As an illustration, we apply our framework to derive new iteration complexity results for different variants of $\SubGrad$ --- including the truncated subgradient, the stochastic subgradient, and the proximal subgradient methods --- when utilized to minimize non-Lipschitz convex / weakly convex objective functions; see \Cref{sec:variants}.

\subsection{Prior Arts}\label{sec:related works}

In this subsection, we provide an overview of the literature on subgradient-type methods when utilized to tackle non-Lipschitz objective functions. 

%Our survey is likely incomplete due to the extremely large volume of existing works on subgradient-type methods. 

There are extensive iteration complexity and convergence results for $\SubGrad$ and its variants. Most of them assume that $f$ is Lipschitz continuous, or equivalently, the subgradients of $f$ are globally bounded from above. Nonetheless, in sharp contrast to recent typical results, early asymptotic convergence analyses of $\SubGrad$ for the convex case, which date back to, e.g., Shor \cite{Shor85}, Polyak \cite{Pol87}, Cohen and Zhu \cite{CohenZ}, and  Alber et al. \cite{alber1998projected}, actually do not impose the Lipschitz continuity assumption. In particular, Shor established certain `hyperplane to optimum distance' convergence result  \cite[Theorem 2.1]{Shor85} and then conducted an involved analysis to show that the sequence of objective function values converges to the optimal value \cite[Theorem 2.2]{Shor85}. Though these results are general enough to cover non-Lipschitz convex functions, they are of asymptotic nature without any rate reported. Later, Nesterov \cite[Theorem 3.2.2]{nesterov2018lectures} and Grimmer  \cite[Section 1.1]{grimmer2019} provided finite-time complexity guarantees for $\SubGrad$ in the convex case based on Shor's `hyperplane to optimum distance' argument.  By crucially utilizing the convexity of the objective function, they transferred the `hyperplane to optimum distance' measure to the suboptimality gap, giving their final complexity results without assuming Lipschitz continuity. Their proofs require fixing an iterate that achieves the minimal `hyperplane to optimum distance' over the first $T$ iterations and then transferring the optimality measure based on this fixed iterate. However, it remains unclear how their results might be generalized to handle the presence of the quadratic growth property or to analyze the proximal version of $\SubGrad$. Compared to these existing results, we propose a different idea that is of plugin nature, therefore --- building on standard analysis --- we can flexibly sharpen these results by invoking additional structures such as strong convexity / quadratic growth, establish similar results for the truncated, stochastic, and proximal subgradient methods, and extend these results to the nonconvex setting. 

There are also several existing works that attempt to relax the Lipschitz continuity assumption by replacing it with less restrictive assumptions,  transforming the non-Lipschitz objective function to some other form, or changing the algorithmic oracle of $\SubGrad$.  Let us now briefly go through some related ones. Cohen and Zhu \cite{CohenZ} proposed a linearly bounded subgradients assumption for convex functions, which generalizes the Lipschitz continuity assumption to allow linear growth of the subgradients and can cover non-Lipschitz convex functions that have a quadratic behavior.  This linearly bounded subgradients assumption was later used by Culioli and Cohen \cite{culioli1990} to analyze the stochastic subgradient method for stochastic convex optimization. Both \cite{CohenZ} and \cite{culioli1990} concern asymptotic convergence properties and do not provide explicit rate result.  Recently, Zhao et al. \cite{zhao2022randomized} imposed this linearly bounded subgradients assumption to derive {iteration} complexity and convergence results for a randomized coordinate subgradient method. To deal with non-Lipschitz convex optimization, {based on the radial transformation framework proposed by Renegar \cite{renegar2016}, Grimmer \cite{grimmer2018} proposed a radial subgradient method, which performs a subgradient step on the radial reformulation of the objective function and then a line search step to maintain feasibility in each iteration}. He then showed that typical {iteration} complexity bounds can be obtained without assuming {the} Lipschitz continuity of the original objective function. On another front, Lu \cite{lu2019} introduced the concept of relative continuity for convex functions, which imposes some relaxed bound rather than a global bound on the subgradients by using a Bregman distance. He then utilized this Bregman distance in algorithm design, resulting in a mirror descent-type algorithm that has a different oracle from $\SubGrad$. {Similar idea was later adopted by Davis et al. \cite{davis2018higher} and Zhou et al.  \cite{zhou2020} for analyzing stochastic model-based methods and online convex optimization algorithms, respectively.} {By contrast, we neither require any modification to $\SubGrad$ nor impose any growth condition on the subgradients.}

\subsection{Basic Notions in Nonsmooth Optimization}\label{sec:prelimilary}
Let us briefly review some basic notions in nonsmooth optimization that will be useful for our subsequent development. For a convex function $\phi:\R^d\rightarrow\R$, its subdifferential is defined as 
\be\label{eq:cvx subdifferential}
\partial \phi(x) := \left\{g \in \R^d: \phi(y)\geq \phi(x)+\langle g,y-x\rangle, \forall y \in \R^d \right\}.
\ee
Any element  $g(x) \in \partial \phi(x)$ is called a subgradient of $\phi$ at $x$.

A  function $\psi:\R^d\rightarrow\R$ is $\rho$-weakly convex for some $\rho\geq 0$ if there exists a convex function $\phi:\R^d\to \R$ such that
$\psi (x) = \phi(x)-\frac{\rho}{2}\|x\|^2$.
For a $\rho$-weakly convex function $\psi$,  its (Fr\'echet) subdifferential is given by \cite[Proposition 4.6]{Vial83}
\[
\partial \psi(x)=\partial \phi(x)-\rho x,
\]
where $\partial \phi$ is the convex subdifferential defined in \eqref{eq:cvx subdifferential}. 
The $\rho$-weak convexity of $\psi$ is equivalent to 
\be\label{eq:wcvx inequality}
\psi(y)\geq \psi(x)+\langle g(x),y-x\rangle-\frac{\rho}{2}\|x-y\|^2
\ee
for all $x,y\in \R^d$ and any subgradient $g(x) \in \partial \psi (x)$ \cite[Proposition 4.8]{Vial83}.

Given a $\rho$-weakly convex function $\psi$, its Moreau envelope and proximal mapping with parameter $\lambda >0$ are defined as \cite{Rockafellar}
\begin{eqnarray}
	&&\psi_{\lambda}(x):=\min_{y\in \R^d}\left\{\psi(y)+\frac{1}{2\lambda}\|y-x\|^2\right\},\label{eq:ME}\\
	&&\prox_{\lambda,\psi}(x):=\argmin_{y\in \R^d}\left\{\psi(y)+\frac{1}{2\lambda}\|y-x\|^2\right\},\label{eq:PM}
\end{eqnarray}
respectively. As a standard requirement on the regularization parameter $\lambda$, we will always assume that $\lambda <\frac{1}{\rho}$ in the sequel, so that $\prox_{\lambda,\psi}(x)$ is unique and $\nabla \psi_{\lambda}$ is well defined.

The following properties of the Moreau envelope are useful for our later development: 
\be\label{eq:envelope properties}
\left[
\begin{aligned}
	&f_{\lambda}(x) \leq f(x)-\frac{1-\lambda\rho}{2\lambda}\|x-\prox_{\lambda,f}(x)\|^2,  \\
	&\nabla f_{\lambda}(x)=\frac{1}{\lambda}(x-\prox_{\lambda,f}(x)), \\
	&\lambda \dist(0,\partial f(\prox_{\lambda,f}(x)))\leq\|x-\prox_{\lambda,f}(x)\|. 
\end{aligned}
\right.
\ee
The first property can be found in \cite[Proposition 4 (i)]{zhudenglizhao2021}. The last two properties are used in \cite{davis2019stochastic} for validating the surrogate stationarity measure $\nabla f_{\lambda}$ for weakly convex optimization.

\section{ Proofs of Complexity Results}\label{sec:SubGrad}
We start by establishing a rather standard preliminary recursion for the iterates of $\SubGrad$, which serves as the starting point of our analyses for both the convex and weakly convex cases.

\begin{lemma}[basic recursion for $\SubGrad$]\label{lemma:inequality_iter} Consider the  step sizes $\alpha_k = \beta_k/\|g(x^k)\|$ with $\beta_k>0$ for all $k\geq 0$. 
Then, for any $x\in \R^d$, we have 
\[
\|x^{k+1}-x\|^2 = \|x^k-x\|^2 - 2\frac{\beta_k}{\|g(x^k)\|}\langle g(x^k),x^{k}-x\rangle +\beta_k^2.
\]
\end{lemma}
\begin{proof}
	By the update formula of $\SubGrad$, we have 
	   \[
		\begin{aligned}
			\|x^{k+1}-x\|^2&=\|x^{k}-x\|^2 - 2\alpha_k \langle g(x^k), x^{k} - x\rangle  + \alpha_k^2 \|g(x^k)\|^2  \\
			&=\|x^{k}-x\|^2 - 2\frac{\beta_k}{\|g(x^k)\|}\langle g(x^k), x^{k} - x\rangle  + \beta_k^2.
		\end{aligned}
	\]
\end{proof}

\subsection{Proof of \Cref{theo:complexity_cvx}} 
By \Cref{lemma:inequality_iter} with  $x=x^*$ and  the convexity of $f$, we have
\be\label{eq:complexity_cvx_1}
\|x^{k+1}-x^*\|^2\leq\|x^{k}-x^*\|^2-\frac{2\beta_k}{\|g(x^k)\|}(f(x^{k})-f^*) +\beta_k^2.
\ee
Since $f(x^k)-f^*\geq0$ and $\beta_k>0$, we have
\[
\|x^{k+1}-x^*\|^2\leq\|x^{k}-x^*\|^2+\beta_k^2.
\]
Unrolling the above inequality up to any $k\leq T$ and using the setting in \eqref{eq:constant step size} yield
\[
\|x^{k+1}-x^*\|^2\leq\|x^{0}-x^*\|^2+\sum_{j=0}^k\beta_j^2\leq\|x^{0}-x^*\|^2+c^2,\quad  0  \leq k\leq T.
\]
This shows that the trajectory $\{x^k\}_{k=0}^T$ lies in the ball $ \calA := \{x\in\R^d: \|x-x^*\|^2 \leq \|x^{0}-x^*\|^2+c^2\}$. 

By \cite[Proposition 4.4]{Vial83}, for any bounded open convex set $\calA'$ that contains $\calA$, the function $f$ is Lipschitz continuous on $\calA'$ with some parameter $\sL_{\calA'}>0$. Using the fact that $\{x^k\}_{k=0}^T$ lies in $\calA$ and \Cref{fact}, we obtain
\be\label{eq:complexity bound subgrad cvx}
\|g(x^k)\|\leq \sL_{\calA'}, \quad  0\leq k  \leq T.
\ee
Next, invoking \eqref{eq:complexity bound subgrad cvx} in \eqref{eq:complexity_cvx_1} provides the recursion
\be\label{eq:recursion subgrad cvx}
\|x^{k+1}-x^*\|^2\leq\|x^{k}-x^*\|^2-\frac{2\beta_k}{\sL_{\calA'}}(f(x^{k})-f^*) +\beta_k^2.
\ee
Unrolling this recursion up to $k=T$ gives 
\[
\sum_{k=0}^T \beta_k (f(x^k)-f^*)
\leq \frac{\sL_{\calA'}}{2}\left(\|x^{0}-x^*\|^2 + \sum_{k=0}^T\beta_k^2\right).
\]
Now, using the definition of $\beta_k$ in \eqref{eq:constant step size}, the definition of $\tilde{x}^T$, and  the convexity of  $f$, we have
\[
\begin{aligned}
	f(\tilde{x}^T)-f^*\leq\sum_{k=0}^T\frac{1}{T+1} (f(x^k)-f^*)
	\leq\frac{\sL_{\calA'}\left(\|x^{0}-x^*\|^2 /c + c\right)}{2\sqrt{T+1}},
\end{aligned}
\]
which yields the desired result.

\subsection{Proof of \Cref{theo:complexity wcvx}}

We start with the following lemma.
\begin{lemma}\label{lem:level set}
	The following two statements are equivalent:
	\begin{enumerate}[label=\textup{\textrm{(\alph*)}},topsep=0pt,itemsep=0ex,partopsep=0ex]
		\item The function $f$ has bounded sublevel sets. 
		\item For any $\lambda >0$, the function $f_\lambda$ has bounded sublevel sets. 
	\end{enumerate}
\end{lemma}

\begin{proof}
	The direction  ``(b) $\Rightarrow$ (a)''  is straightforward since for any $\lambda >0$, we have $f(x) \geq f_\lambda(x)$ for all  $x\in \R^d$. Next, we consider the opposite direction. For any $\lambda > 0$ and $s \in \mathbb{R}$, let $x$ be such that
	\[
	     f_{\lambda} (x) = f(\widehat x) + \frac{1}{2\lambda} \|\widehat x - x\|^2  \leq s, 
	\]
	where $\widehat x:=\prox_{\lambda,f}(x)$. Since $f$ has bounded sublevel sets, it follows from $f(\widehat x)  \leq s$ that $\widehat x$ lies in the compact set $\calS_{\leq s}:=\{x\in\R^d: f(x) \leq s\}$. Hence, there exists a $D>0$ such that $\|\widehat x - z\|\leq D$ for any $z\in \calS_{\leq s}$. This, together with $\frac{1}{2\lambda}\|\widehat x - x\|^2  \leq s$, gives $\|x-z\| \leq \|\widehat x - x\| + \|\widehat x - z\| \leq \sqrt{2\lambda s} + D $ for any $z \in \mathcal{S}_{\le s}$. This implies that $x$ lies in the compact set $\left\{ u \in \mathbb{R}^d : \min_{z \in \mathcal{S}_{\le s}} \| u - z \| \le \sqrt{2\lambda s} + D \right\}$.
\end{proof}

We now proceed to the proof of \Cref{theo:complexity wcvx}.

\begin{proof}[Proof of \Cref{theo:complexity wcvx}]
	Let us use $\widehat{x}^k$ to denote $\prox_{\lambda,f}(x^k)$ in the following derivations. 
	By  \Cref{lemma:inequality_iter} (set $x=\widehat{x}^k$) and the $\rho$-weak convexity of $f$ \eqref{eq:wcvx inequality}, we have
	\be\label{eq:complexity recur-wcvx1}
	\|x^{k+1} - \widehat{x}^k\|^2\leq \|x^{k}-\widehat{x}^k\|^2-\frac{2\beta_k}{\|g(x^k)\|}\left(f(x^k)-f(\widehat{x}^k)-\frac{\rho}{2}\|x^k - \widehat{x}^k\|^2\right) +\beta_k^2.
	\ee
	We then follow the technique developed in \cite{davis2019stochastic} to derive {an initial} recursion for the {trajectory generated by} $\SubGrad$. The definitions of the Moreau envelope~\eqref{eq:ME} and the proximal mapping~\eqref{eq:PM} imply that
	\[
	f_{\lambda}(x^{k+1})\leq f(\widehat{x}^k)+\frac{1}{2\lambda}\|\widehat{x}^k-x^{k+1}\|^2.
	\]
	Combining the above two inequalities yields
	\be\label{eq:complexity bound2-1_wcvx}
	\begin{aligned}
		f_{\lambda}(x^{k+1})&\leq f(\widehat{x}^k)+\frac{1}{2\lambda}\|\widehat{x}^k-x^{k}\|^2 - \frac{\beta_k}{\lambda\|g(x^k)\|}\left(f(x^k)-f(\widehat{x}^k)-\frac{\rho}{2}\|x^k-\widehat{x}^k\|^2\right) +\frac{\beta_k^2}{2\lambda}\\
%		&=f_{\lambda}(x^k) -\frac{\beta_k}{\lambda\|g(x^k)\|}\left(f(x^k)-f(\widehat{x}^k)-\frac{\rho}{2}\|x^k-\widehat{x}^k\|^2\right) + \frac{\beta_k^2}{2\lambda} \\
		&= f_{\lambda}(x^k) -\frac{\beta_k}{\lambda\|g(x^k)\|}\left(f(x^k)-f_\lambda({x}^k) + \left(\frac{1}{2\lambda}-\frac{\rho}{2}\right)\|x^k-\widehat{x}^k\|^2\right) + \frac{\beta_k^2}{2\lambda}\\
		&\leq f_{\lambda}(x^k) -\frac{(1/\lambda - \rho)\beta_k}{\lambda\|g(x^k)\|}\|x^k-\widehat{x}^k\|^2 + \frac{\beta_k^2}{2\lambda},
	\end{aligned}
	\ee
	where we have used the first property of \eqref{eq:envelope properties} in the last inequality. This further implies that 
	\[
	f_{\lambda}(x^{k+1})\leq f_{\lambda}(x^k) + \frac{\beta_k^2}{2\lambda}.
	\]
	Unrolling this inequality up to any $k\leq T$ and invoking the condition on $\beta_k$ in \eqref{eq:constant step size} give
	\be\label{eq:complexity bound moreau}
	    f_{\lambda}(x^{k+1})\leq f_{\lambda}(x^0) + \sum_{j=0}^k \frac{\beta_j^2}{2\lambda}\leq f_{\lambda}(x^0) + \frac{c^2}{2\lambda}, \quad 0\leq  k \leq T.
	\ee
	By \Cref{lem:level set}, since $f$ has bounded sublevel sets,  so does $f_\lambda$. It then immediately follows from \eqref{eq:complexity bound moreau} that the trajectory $\{x^k\}_{k= 0}^T$ lies in the bounded sublevel set $\calB := \{x\in \R^d: f_\lambda (x) \leq f_{\lambda}(x^0) + c^2/2\lambda\}$. 	
	
	Next, since $\{x^k\}_{k= 0}^T$ lies in $\calB$, \Cref{fact} and  \cite[Proposition 4.4]{Vial83}  imply that for any bounded open convex set $\mathcal{B}'$ that contains $\mathcal{B}$, we have
	\be\label{eq:complexity wcvx subgrad bound}
	\|g(x^k)\|\leq \sL_{\calB'}, \quad 0\leq  k \leq T,
	\ee
	where $\sL_{\calB'} > 0$ is the Lipschitz constant of $f$ on $\calB'$. Plugging \eqref{eq:complexity wcvx subgrad bound} into \eqref{eq:complexity bound2-1_wcvx} gives the recursion 
	\[
	\begin{aligned}
		f_{\lambda}(x^{k+1})\leq f_{\lambda}(x^k) -\frac{(1/\lambda - \rho)}{\lambda \sL_{\calB'}}\beta_k\|x^k-\widehat{x}^k\|^2 + \frac{\beta_k^2}{2\lambda}.
	\end{aligned}
	\]
	Unrolling this recursion up to $k = T$, using \revise{$\nabla f_{\lambda}(x^k)=(x^k - \widehat x^k)/\lambda$ in} \eqref{eq:envelope properties},  rearranging the terms, and noting that $f_\lambda(x) = f(\prox_{\lambda,f}(x)) + \frac{1}{2\lambda}\|\prox_{\lambda,f}(x) -x\|^2 \geq f(\prox_{\lambda,f}(x))  \geq f^*$ for all $x\in \R^d$, we obtain
	\[
	\min_{0\leq k\leq T}\|\nabla f_{\lambda}(x^k)\|^2\leq\frac{\sL_{\calB'}\left(f_{\lambda}(x^0)-f^*+\frac{1}{2\lambda}\sum_{k=0}^T\beta_k^2\right)}{(1-\lambda\rho)\sum_{k=0}^{T}\beta_k}.
	\]
	The  complexity bound \eqref{eq:wcvx complexity} readily follows by plugging the specific formulae of $\beta_k$ in \eqref{eq:constant step size} into the above bound. 
\end{proof}

\section{Convergence Results with Diminishing Step Size Rules}\label{sec:convergence}
In the previous section, we introduced the finite-time complexity results for $\SubGrad$.  They are integral to understanding the core properties and progress of the method during the pre-determined $T$ iterations. On the other hand, the convergence behavior of the method plays an equally important role, as it characterizes whether the method can eventually approach an exact global minimizer (for convex case) / stationary point (for the weakly convex case) or not.  Therefore, we now turn to investigate the convergence behavior of $\SubGrad$.

\subsection{Convergence Results for Convex Case}

\revise{By using a proper diminishing step size rule, we have the following corollary of \Cref{theo:complexity_cvx}, which concerns the convergence rate of $\SubGrad$ when the total number of iterations $T$ is not known a priori.}

\begin{corollary}[convergence rate of $\SubGrad$ for convex case]\label{theo:rate_cvx} Suppose that $f$ in \eqref{eq:problem} is convex, and the step sizes $\{\alpha_k\}$ satisfy
	\be\label{eq:step sizes}
	\alpha_k = \frac{\beta_k}{\|g(x^k)\|}  \quad \text{with} \quad \beta_k>0 \quad \text{and} \quad  \sum_{k=0}^\infty \beta_k^2 \leq \mathfrak{b} <\infty.
	\ee
	Let $\tilde{x}^k=\frac{\sum_{j=0}^k\beta_j x^j}{\sum_{j=0}^k\beta_j}$ for all $k\geq 0$ and $x^*\in \calX^*$ be some fixed optimal solution to problem \eqref{eq:problem}. Then,  the trajectory $\{x^k\}$ lies in the ball $\mathcal{C} := \{x\in\R^d: \|x-x^*\|^2 \leq \|x^{0}-x^*\|^2+ \mathfrak{b}\}$. Consequently, for any bounded open convex set $\mathcal{C}'$ that contains $\mathcal{C}$, we have
	\be\label{eq:rate_cvx}
	   f(\tilde{x}^{k})-f^*\leq \frac{\sL_{\mathcal{C}'}\left(\|x^0-x^*\|^2 + \mathfrak{b}\right)}{2\sum_{j=0}^{k}\beta_j}, \quad \forall k \geq 0,
	\ee
	where $\sL_{\mathcal{C}'}>0$ is the Lipschitz constant of $f$ on $\mathcal{C}'$.
	Moreover, if $\beta_k = \frac{c}{\sqrt{k+1}\log(k+2)}$ for all $k\ge0$, where $c>0$ is some constant, then we have
	\be\label{eq:rate_cvx_sqrt}
	f(\tilde{x}^{k})-f^*\leq \frac{\sL_{\mathcal{C}'}\left(\|x^0-x^*\|^2 /2c + c/\log^2 2\right) \log(k+2)}{\sqrt{k+1}}, \quad \forall k\geq 0.
	\ee
\end{corollary}

\begin{proof} \revise{Under the step size rule \eqref{eq:step sizes},  following exactly the same proof as that of \Cref{theo:complexity_cvx} gives}
%	By letting $x = x^*$ in \Cref{lemma:inequality_iter} and the convexity of $f$, we have
%	\be\label{eq:complexity_cvx_1-1}
%	\|x^{k+1}-x^*\|^2\leq\|x^{k}-x^*\|^2-\frac{2\beta_k}{\|g(x^k)\|}(f(x^{k})-f^*) +\beta_k^2.
%	\ee
%	 This gives $\|x^{k+1}-x^*\|^2\leq\|x^{k}-x^*\|^2+\beta_k^2$. Unrolling this inequality to any $k\geq 0$, together with the last condition on $\{\beta_k\}$ in \eqref{eq:step sizes}, yields
	\[
	\|x^{k+1}-x^*\|^2\leq \|x^0 - x^*\|^2 +\sum_{j=0}^k\beta_j^2\leq\|x^0 - x^*\|^2 + \mathfrak{b}, \quad \forall k\geq 0.
	\]
	This implies that the trajectory $\{x^k\}$ lies in the ball $\mathcal{C} := \{x\in\R^d: \|x-x^*\|^2 \leq \|x^{0}-x^*\|^2+ \mathfrak{b}\}$. It then follows from \Cref{fact} and \cite[Proposition 4.4]{Vial83} that for any bounded open convex set $\calC'$ that contains $\calC$, we have $\|g(x^k)\|\leq \sL_{\calC'}$ for all $k\geq 0$, where $\sL_{\mathcal{C}'}>0$ is the Lipschitz constant of $f$ on $\calC'$. Plugging this bound into \revise{\Cref{lemma:inequality_iter} with $x = x^*$ and utilizing the convexity of $f$}, we obtain the following recursion: 
	\be\label{eq:rate_cvx_4}
	\|x^{k+1}-x^*\|^2\leq\|x^{k}-x^*\|^2-\frac{2\beta_k}{\sL_{\mathcal{C}'}}(f(x^{k})-f^*) +\beta_k^2,  \quad \forall k\geq 0.
	\ee
	 Unrolling the recursion in \eqref{eq:rate_cvx_4},  together with the definition of $\tilde{x}^k$ and  convexity of  $f$, we have
	\[
	\begin{aligned}
		f(\tilde{x}^k)-f^*\leq\sum_{j=0}^k\frac{\beta_j }{\sum_{j=0}^{k}\beta_j} (f(x^j)-f^*)
		\leq\frac{\sL_{\mathcal{C}'}\left(\|x^0 - x^*\|^2 + \mathfrak{b}\right)}{2\sum_{j=0}^{k}\beta_j},  \quad \forall k\geq 0,
	\end{aligned}
	\]
	which establishes \eqref{eq:rate_cvx}. 
	
	We now derive the convergence rate in \eqref{eq:rate_cvx_sqrt} with $\beta_k = \frac{c}{\sqrt{k+1}\log(k+2)}$ for all $k\ge0$, where $c>0$ is some constant.  By the integral comparison test, we have
	\[
	\sum_{j=0}^k\beta_j^2 \leq  \beta_0^2 + \int_{0}^{k} \frac{c^2}{(t+1)\log^2(t+2)} dt \leq \frac{2c^2}{\log^2 2} =: \mathfrak{b}.
	\]
	Thus, all the conditions in \eqref{eq:step sizes} are verified. 
	Invoking the above upper bound and the fact that $\sum_{j=0}^{k}\beta_j \geq \sum_{j=0}^{k} \frac{c}{\sqrt{k+1}\log(k+2)} = \frac{c \sqrt{k+1}}{\log(k+2)}$ in \eqref{eq:rate_cvx} yields the desired result. 
\end{proof}

%The convergence result in \Cref{theo:rate_cvx} is often called sub-optimal due to the term ``$\log(k+2)$''. It is possible to remove this log term by pre-determining the total number of iterations $T$ and setting $\beta_k= {c}/{\sqrt{T+1}}$ for  $k = 0, 1,\ldots, T$. In this case, inequality \eqref{eq:rate_cvx_1} still holds and the bound on $\{x^k\}_{k=0}^T$ in \eqref{eq:rate_cvx_3} becomes 
%\[
%\|x^{k}\|\leq \|x^*\|+\sqrt{\dist^2(x^{0},\calX^*)+ c^2} \quad \text{for} \ k= 0, 1,\ldots, T.
%\]
%Thus, we can derive complexity result for $\SubGrad$ by fully following the arguments of  \Cref{theo:rate_cvx}, resulting in the following result . 
%
%Though the result in \Cref{theo:complexity_cvx} does not have a ``$\log(T+2)$'' term, it is a complexity-type result. In particular, we have to pre-determine the total number of iterations $T$, which needs to be finite. Furthermore, we need to set $\beta_k$ to be a constant that can be relatively small when $T$ is large. 
\paragraph{Recovery of Shor's convergence result.} With the additional typical requirement $\sum_{k=0}^{\infty} \beta_k = \infty$ on $\{\beta_k\}$, we can recover Shor's asymptotic convergence result in \cite[Theorem 2.2]{Shor85}.  Towards that end, we first observe that the recursion \eqref{eq:rate_cvx_4} and the last condition on $\{\beta_k\}$ in \eqref{eq:step sizes} imply $\sum_{k=0}^{\infty} \beta_k(f(x^{k})-f^*)<\infty$. Now, applying \cite[Lemma 4]{CohenZ} with $ \sum_{k=0}^\infty\beta_k =\infty$ and invoking the fact that $f$ is Lipschitz continuous on $\calC'$, we conclude that $\lim_{k\rightarrow\infty}f(x^k) = f^*$. 

We note that some convex learning problems (e.g., the SVM \eqref{eq:svm}) exhibit benign problem structures, such as strong convexity. In the following, we derive an enhanced convergence rate of $\SubGrad$ when utilized to solve convex optimization problems with the quadratic growth property. Such a property includes strong convexity as a special case. 

\begin{corollary}[convergence rate of $\SubGrad$ for convex case with quadratic growth property]\label{theo:rate_qg}
Suppose that $f$ in \eqref{eq:problem} is convex and possesses the quadratic growth property with parameter $\mu>0$ (i.e., there exists a constant $\mu>0$ such that $f(x)-f^*\geq \mu\dist^2(x,\calX^*)$ for all $x\in\setR^d$). Suppose further that the step sizes $\{\alpha_k\}$ satisfy
\be\label{eq:step sizes qg}
\alpha_k=\frac{\beta_k}{\|g(x^k)\|} \quad  \text{with} \quad  \beta_k = \frac{\sL_{\mathcal{\calC'}}}{\mu({k+1})}, \quad  \forall k\geq 0.
\ee
Then, we have
\be\label{eq:qg rate}
{\dist}^2\left({x}^{k+1},\calX^*\right) \leq \frac{\sL_{\mathcal{\calC'}}^2}{\mu^2(k+1)}, \quad \forall k \geq 0.
\ee
Here, the Lipschitz constant $\sL_{\calC'}$ of $f$ is defined in \Cref{theo:rate_cvx}.
\end{corollary}

\begin{proof}
	Note that the step size rule defined in \eqref{eq:step sizes qg} satisfies the conditions in \eqref{eq:step sizes}. Therefore, following the proof of \Cref{theo:rate_cvx}, we can show that the trajectory $\{x^k\}$ lies in $\calC'$ and $\|g(x^k)\|\leq \sL_{\mathcal{\calC'}}$ for all $k \geq 0$. 
    Then, we invoke this upper bound and  set $x = \proj_{\calX^*}(x^k)$ in \Cref{lemma:inequality_iter} to obtain 
	\[
       \dist^2(x^{k+1},\calX^*) \leq \dist^2(x^{k},\calX^*)-\frac{2\beta_k}{\sL_{\mathcal{\calC'}}}(f(x^{k})-f^*) +  \beta_k^2, \quad \forall k\geq 0.
	\]
%	Plugging the quadratic growth condition in the above inequality provides 
%	\[
%	\dist^2(x^{k+1},\calX^*)\leq \dist^2(x^{k},\calX^*) -\frac{2\mu\beta_k}{\sL_{\mathcal{B}_1}}\dist^2(x^k,\calX^*) + \beta_k^2\quad \forall k\geq 0.
%	\]
	Plugging the quadratic growth property and the definition of $\beta_k$ into the above inequality gives 
	\[
	\dist^2(x^{k+1},\calX^*)\leq \dist^2(x^{k},\calX^*) -\frac{2}{k+1}\dist^2(x^k,\calX^*) + \frac{\sL_{\calC'}^2}{\mu^2(k+1)^2}, \quad \forall k\geq 0.
	\]
	By multiplying both sides of the above inequality by $(k+1)^2$, we obtain 
	\[
	(k+1)^2\dist^2(x^{k+1},\calX^*)\leq \left((k+1)^2-2k-2\right)\dist^2(x^{k},\calX^*) + \frac{\sL_{\calC'}^2}{\mu^2}, \quad \forall k\geq 0.
	\]
	This, together with the inequality $(k+1)^2-2k-2=k^2-1< k^2$, implies that
	\be\label{eq:rate_qg}
	(k+1)^2\dist^2(x^{k+1},\calX^*) -  k^2\dist^2(x^{k},\calX^*)\leq \frac{\sL_{\calC'}^2}{\mu^2}, \quad \forall k\geq 0.
	\ee
	Next, we obtain
	\[
	\begin{aligned}
		(k+1)^2\dist^2(x^{k+1},\calX^*)&=\sum_{j=0}^k\left[(j+1)^2\dist^2(x^{j+1},\calX^*)-j^2\dist^2(x^{j},\calX^*)\right] \leq \frac{\sL_{\calC'}^2}{\mu^2}(k+1),
	\end{aligned}
	\]
	where the inequality follows by applying \eqref{eq:rate_qg} $k+1$ times. 
	Upon dividing both sides by $(k+1)^2$, we finally establish the desired convergence rate result
	\[
	\dist^2(x^{k+1},\calX^*)\leq\frac{\sL_{\calC'}^2}{\mu^2(k+1)}.
	\]
\end{proof}

Some remarks on \Cref{theo:rate_qg} are in order. One can observe that the right-hand side of \eqref{eq:qg rate} is independent of the initial distance $\dist(x^0, \calX^*)$, which seems unreasonable. Actually, a quick derivation shows that $\dist(x^k, \calX^*) \leq \sL_{\calC'}/\mu$ for all $k\geq 0$. To see this, by utilizing the convexity and quadratic growth property of $f$, we have
\[
   \mu \dist^2(x, \calX^*) \leq f(x) - f^* \leq \langle g(x), x-x^* \rangle  \leq \|g(x)\| \cdot  \dist(x, \calX^*) 
\]
for any $x\in \R^d$, where $x^* = \proj_{\calX^*}(x)$. This gives 
\be\label{eq:subgrad lower bound}
    \dist(x, \calX^*) \leq \frac{\|g(x)\|}{\mu}, \quad \forall x\in \R^d.
\ee
If we set $x = x^k$, then we can invoke the derived upper bound $\|g(x^k)\| \leq \sL_{\calC'}$ for all $k \geq 0$ in the above inequality to obtain 
\[
  \dist(x^k, \calX^*) \leq \frac{\sL_{\calC'}}{\mu}, \quad \forall k \geq 0.
\]
This explains why the bound \eqref{eq:qg rate} can be independent of $\dist(x^0, \calX^*)$.

We remark that it is unrealistic to impose both Lipschitz continuity and quadratic growth / strong convexity assumptions on $f$, as the subgradient norm $\|g(x)\|$ grows at least linearly in $\dist(x, \calX^*)$ as shown in \eqref{eq:subgrad lower bound}. This observation demonstrates the significance of our result for non-Lipschitz convex minimization with the quadratic growth property.

\subsection{Convergence Result for Weakly Convex Case}

We now proceed to the asymptotic convergence result for weakly convex optimization without assuming the Lipschitz continuity of the objective function. 
\begin{theorem}[convergence of $\SubGrad$ for weakly convex case]\label{theo:rate_wcvx} Suppose that $f$ in \eqref{eq:problem} is $\rho$-weakly convex and has bounded sublevel sets. Suppose further that the step sizes $\{\alpha_k\}$ satisfy 
\be\label{eq:step sizes wcvx}
	\alpha_k = \frac{\beta_k}{\|g(x^k)\|}  \quad \text{with} \quad \beta_k>0, \quad  \sum_{k=0}^\infty \beta_k = \infty, \quad \text{and} \quad  \sum_{k=0}^\infty \beta_k^2 \leq \mathfrak{b} <\infty.
\ee
Then, the trajectory $\{x^k\}$ lies in the bounded sublevel set  $\calD:= \{x\in \R^d: f_\lambda (x) \leq f_{\lambda}(x^0) + {\mathfrak{b}}/{2\lambda}\}$. Consequently, we have $\lim_{k\to \infty} \|\nabla f_{\lambda}(x^k)\| = 0$, namely, every accumulation point of $\{x^k\}$ is a stationary point of problem \eqref{eq:problem}. 
\end{theorem}

\begin{proof}
Under the setting of \Cref{theo:rate_wcvx}, we can follow the first several steps in the proof of \Cref{theo:complexity wcvx} to show that the inequality \eqref{eq:complexity bound2-1_wcvx} holds, i.e.,
\be\label{eq:complexity bound2-1_wcvx-1}
	f_{\lambda}(x^{k+1}) \leq f_{\lambda}(x^k) -\frac{(1/\lambda - \rho)\beta_k}{\lambda\|g(x^k)\|}\|x^k-\widehat{x}^k\|^2 + \frac{\beta_k^2}{2\lambda}.
\ee
By unrolling this inequality to any $k\geq 0$, we obtain
\[
      f_{\lambda}(x^{k+1})\leq f_{\lambda}(x^0) + \sum_{j=0}^k \frac{\beta_j^2}{2\lambda}\leq f_{\lambda}(x^0) + \frac{\mathfrak{b}}{2\lambda}, \quad \forall k\geq 0,
\]
where we have used the last condition on $\{\beta_k\}$ in \eqref{eq:step sizes wcvx}. This shows that the trajectory $\{x^k\}$ lies in the bounded sublevel set $\calD:= \{x\in \R^d: f_\lambda (x) \leq f_{\lambda}(x^0) + \frac{\mathfrak{b}}{2\lambda}\}$. Then, for any bounded open convex set $\calD'$ that contains $\calD$, it follows from \Cref{fact} and \cite[Proposition 4.4]{Vial83} that $\|g(x^k)\|\leq \sL_{\calD'}$ for all $k\geq 0$, 
where $\sL_{\calD'} > 0$ is the Lipschitz constant of $f$ on $\calD'$. Combining this upper bound, \eqref{eq:complexity bound2-1_wcvx-1}, and the second property of \eqref{eq:envelope properties} gives 
\[
	f_{\lambda}(x^{k+1})\leq f_{\lambda}(x^k) -\frac{(1 - \lambda \rho)}{ \sL_{\calD'}}\beta_k\|\nabla f_{\lambda}(x^k)\|^2  + \frac{\beta_k^2}{2\lambda}, \quad \forall k \geq 0.
\]
By summing this recursion and using the last condition on $\{\beta_k\}$ in \eqref{eq:step sizes wcvx} and the fact that $f_\lambda (x) \geq f^*$ for all $x\in \R^d$ (derived in the proof of \Cref{theo:complexity wcvx}), we obtain  
\be\label{eq:cd3}
\sum_{k=0}^{\infty} \beta_k \|\nabla f_{\lambda}(x^k)\|^2 <\infty.
\ee
In addition, according to the update of $\SubGrad$, we have
\be\label{eq:cd1}
   \|x^{k+1} - x^k\| = \beta_k, \quad \forall k \geq 0.
\ee
With the developed machineries, we can now derive the asymptotic convergence result by mimicking the construction in \cite[Lemma 4]{CohenZ}. For any $\varepsilon>0$, let $\mathbb{N}_{\varepsilon} :=\{k\in\mathbb{N}:\|\nabla f_{\lambda}(x^k)\|\leq\varepsilon\}$. It immediately follows from  \eqref{eq:step sizes wcvx} and \eqref{eq:cd3} that $\liminf_{k\rightarrow\infty} \|\nabla f_{\lambda}(x^k)\| = 0 $, which implies that $\mathbb{N}_{\varepsilon}$ is an infinite set. Let $\bar{\mathbb{N}}_{\varepsilon}$ be the complementary set of $\mathbb{N}_{\varepsilon}$ in $\mathbb{N}$. Then, we have
\begingroup
\allowdisplaybreaks
\[
\begin{aligned}
	\varepsilon^2\sum_{k\in\bar{\mathbb{N}}_{\varepsilon}}\beta_k&\leq \sum_{k\in\bar{\mathbb{N}}_{\varepsilon}}\beta_k\|\nabla f_{\lambda}(x^k)\|^2
	\leq\sum_{k\in\mathbb{N}}\beta_k\|\nabla f_{\lambda}(x^k)\|^2
	<\infty,
\end{aligned}
\]
\endgroup
where the last inequality is from \eqref{eq:cd3}.
Hence, for any $\delta>0$, there exists an integer $\mathfrak{n}(\delta)$ such that
\[
\sum_{\substack{\ell\geq\mathfrak{n}(\delta),  \ \ell\in\bar{\mathbb{N}}_{\varepsilon}}}\beta_\ell\leq\delta.
\]
Note that $\nabla f_{\lambda}$ is $\sL_{\lambda}$-Lipschitz continuous for some $\sL_\lambda >0$, since $f$ is weakly convex \cite[Corollary 3.4]{hoheisel2020}. Next, we take an arbitrary $\gamma>0$ and set $\varepsilon=\frac{\gamma}{2}$, $\delta=\frac{\gamma}{2\sL_{\lambda}}$. For all $k\geq\mathfrak{n}(\delta)$, if $k\in\mathbb{N}_{\varepsilon}$, then we have $\|\nabla f_{\lambda}(x^k)\|\leq\varepsilon<\gamma$; otherwise $k\in\bar{\mathbb{N}}_{\varepsilon}$. Let $m$ be the smallest element in the set $\{\ell\in\mathbb{N}_{\varepsilon}:\ell\geq k\}$. Clearly, $m$ is finite since $\mathbb{N}_{\varepsilon}$ is an infinite set. Without loss of generality, we can assume that $m>k$.  Then, by the $\sL_{\lambda}$-Lipschitz continuity of $\nabla f_{\lambda}$, \eqref{eq:cd1}, and the definitions of $\varepsilon$, $\delta$, and $m$, we can compute
\[
\begin{aligned}
	\|\nabla f_{\lambda}(x^k)\|& \leq\|\nabla f_{\lambda}(x^k)-\nabla f_{\lambda}(x^m)\|+\|\nabla f_{\lambda}(x^m)\| \\
	&\leq \sL_{\lambda}\sum_{\ell=k}^{m-1}\|x^{\ell+1}-x^\ell\|+\varepsilon 
	\leq \sL_{\lambda}\sum_{\ell=k}^{m-1}\beta_\ell+\frac{\gamma}{2}\\
	&=  \sL_{\lambda}\sum_{k \leq \ell <m, \ \ell\in \bar{\mathbb{N}}_{\varepsilon}}\beta_\ell+\frac{\gamma}{2}
	\leq \sL_{\lambda}\sum_{\ell\geq\mathfrak{n}(\delta), \  \ell\in\bar{\mathbb{N}}_{\varepsilon}}\beta_\ell+\frac{\gamma}{2}
	\leq\gamma.
\end{aligned}
\]
\revise{Here, the equality is due to the fact that for all $\ell$ with $k\leq \ell <m$, we have $\ell\in \bar N_\varepsilon$ by the choice of $m$.} Since $\gamma>0$ is arbitrary, we have $\|\nabla f_{\lambda}(x^k)\|\rightarrow 0$ as $k\rightarrow \infty $.
\end{proof}

\section{Applications to Some Variants of the Subgradient Method}\label{sec:variants}

To illustrate the wide applicability of our framework, we conduct finite-time complexity analysis for some variants of $\SubGrad$ including the truncated subgradient, stochastic, and proximal subgradient methods.  Convergence results for these variants can be readily established by combining the analyses in this section and the techniques in \Cref{sec:convergence}. We shall omit these derivations to avoid repetition.

\subsection{Truncated Subgradient Method}\label{sec:tsubgrad} The work \cite{asi2019stochastic} proposed the truncated subgradient method (henceforth, $\TSM$), which involves a truncation process for the linear model used by $\SubGrad$ if  the optimal objective function value $f^*$  is known. In this subsection, our analysis only requires that a lower bound of $f$ is available, as we do not focus on providing linear convergence under the sharpness condition. Without loss of generality, we can assume that $0$ is the lower bound, i.e., $f(x) \geq 0$ for all $x\in \R^d$, as we can subtract the non-zero lower bound from $f$ if it is known. This is the case for many machine learning and signal recovery problems, as they aim to fit the data based on some nonnegative loss and $0$ is a clear lower bound; see, e.g., the two examples \eqref{eq:svm} and \eqref{eq:rms}. $\TSM$ for solving problem \eqref{eq:problem} has the following update:
\be\label{eq:truncated SubGrad}
    x^{k+1} = \argmin_{x\in \R^d} \ \max\left\{ f(x^k) + \left\langle g(x^k), x- x^k\right\rangle , 0 \right\} + \frac{1}{2\alpha_k} \|x - x^k\|^2.
\ee
Since $\TSM$ utilizes more information about $f$, it is more robust to the choices of step sizes compared to $\SubGrad$ as shown experimentally in \cite{asi2019stochastic}. Theoretically, we can follow our framework to show that  $\TSM$  has the same complexity as those derived in the previous sections for $\SubGrad$.
\begin{lemma}\label{lemma:recursion TSubGrad}
   Suppose that $f$ in \eqref{eq:problem} satisfies $f(x) \geq 0$ for all $x\in \R^d$. Consider $\TSM$ \eqref{eq:truncated SubGrad} with step sizes $\alpha_k = \beta_k/\|g(x^k)\|$, where $\beta_k>0$, for all $k\geq 0$. 
   Then, for any $x\in \R^d$, the following results hold:
   \begin{enumerate}[label=\textup{\textrm{(\alph*)}},topsep=0pt,itemsep=0ex,partopsep=0ex]
   	\item If $f$ is convex, then 
   	\[
   	\|x^{k+1} - x\|^2 \leq \|x^k - x\|^2 - 2 \frac{\beta_k}{\|g(x^k)\|} \left(  f(x^k) - f(x)\right)  + \beta_k^2.
   	\]
   	\item If $f$ is $\rho$-weakly convex, then
   	\[
   	\|x^{k+1} - x\|^2 \leq \|x^k - x\|^2 - 2 \frac{\beta_k}{\|g(x^k)\|} \left(  f(x^k) - f(x) - \frac{\rho}{2} \|x^k - x\|^2\right)  + \beta_k^2.
   	\]
   \end{enumerate}
\end{lemma}

\begin{proof}
	By the $1/\alpha_k$-strong convexity of the subproblem \eqref{eq:truncated SubGrad} and the fact that $x^{k+1}$ is its solution, we have 
	\[
	\begin{split}
		&\max\left\{ f(x^k) + \left\langle g(x^k), x- x^k\right\rangle , 0 \right\} + \frac{1}{2\alpha_k} \|x - x^k\|^2 \\
		&\geq \max\left\{ f(x^k) + \left\langle g(x^k), x^{k+1}- x^k\right\rangle , 0 \right\} + \frac{1}{2\alpha_k} \|x^{k+1} - x^k\|^2 +  \frac{1}{2\alpha_k} \|x^{k+1} - x\|^2.
	\end{split}
	\]
	When $f$ is $\rho$-weakly convex, the nonnegativity of $f$ gives 
	\[
	\max\left\{ f(x^k) + \left\langle g(x^k), x- x^k\right\rangle , 0 \right\} \leq \max\left\{ f(x) +  \frac{\rho}{2} \|x^k - x\|^2, 0 \right\}  \leq f(x) +  \frac{\rho}{2} \|x^k - x\|^2.
	\]
	Upon combining the above two bounds, multiplying both sides by $2\alpha_k$, and rearranging terms, we obtain 
	\[
	     \|x^{k+1} - x\|^2 \leq \|x^k - x\|^2 - 2 \frac{\beta_k}{\|g(x^k)\|} \left(  f(x^k) - f(x) - \frac{\rho}{2} \|x^k - x\|^2\right)  + \beta_k^2, 
	\]
	where we have used the definition of $\alpha_k$ and Young's inequality. This recursion yields the result in (b).  When $f$ is convex, setting $\rho = 0$ in the above inequality yields the result in (a).
\end{proof}

\Cref{lemma:recursion TSubGrad} reveals that $\TSM$ has exactly the same recursions as those of $\SubGrad$ for both the convex and weakly convex cases (cf. \eqref{eq:complexity_cvx_1} and  \eqref{eq:complexity recur-wcvx1}) when $x$ is set appropriately. Therefore, $\TSM$  shares the same finite-time complexity guarantees as those for $\SubGrad$ established in \Cref{theo:complexity_cvx} and \Cref{theo:complexity wcvx}. We summarize these guarantees in the following corollary. 
\begin{corollary}\label{cor:compelxity TSubGrad}
	Suppose that $f$ in \eqref{eq:problem} satisfies $f(x) \geq 0$ for all $x\in \R^d$. Consider $\TSM$ \eqref{eq:truncated SubGrad} with step sizes $\{\alpha_k\}_{k= 0}^T$ satisfying the conditions in \eqref{eq:constant step size}. Then, the following complexity bounds hold:
	 \begin{enumerate}[label=\textup{\textrm{(\alph*)}},topsep=0pt,itemsep=0ex,partopsep=0ex]
	 	\item Suppose further that $f$ is convex. 	Let $\tilde{x}^T=\frac{1}{T+1}\sum_{k=0}^T x^k$ and $x^*\in \calX^*$ be some fixed optimal solution to problem \eqref{eq:problem}. Then, we have
	 	\[
	 	     f(\tilde{x}^{T})-f^*\leq \frac{\sL_{\mathcal A'}\left(\|x^0-x^*\|^2/c + c\right)}{2\sqrt{T+1}},
	 	\]
	 	where the Lipschitz constant ${\sf L}_{\mathcal A'}$ of $f$ on $\calA'$ is defined in \Cref{theo:complexity_cvx}.
	 	\item Suppose further that $f$ is $\rho$-weakly convex and has bounded sublevel sets.  Then, we have 
	 	\[
	 	\min_{0\leq k\leq T}\|\nabla f_{\lambda}(x^k)\|^2 \leq \frac{\sL_{\calB'}\left(f_{\lambda}(x^0)-f^*+ c^2 / 2\lambda \right)}{c(1-\lambda\rho)\sqrt{T+1}}.
	 	\]
	 	where the Lipschitz constant $\sL_{\calB'} $ of $f$ on $\calB'$ is defined in \Cref{theo:complexity wcvx}.
	 \end{enumerate}
\end{corollary}

\subsection{Stochastic Subgradient Method}\label{sec:stochastic incremental methods}
In this subsection, we assume that $f$ in problem \eqref{eq:problem} has the finite-sum structure
\be\label{eq:finite sum}
f(x)  = \frac{1}{n} \sum_{i = 1}^{n} f_i(x).
\ee
Here, each component function $f_i$ is assumed to be convex. To simplify the analysis, let us impose the interpolation condition  \cite{belkin2018overfitting,vaswani2019fast}, which implies that each component function shares the same minimizers as $f$. Such a condition is   termed an "easy" problem class in \cite{asi2019stochastic}. 

\begin{definition}[interpolation condition]\label{def:easy}
	The function $f$ is said to satisfy the interpolation condition if $\min_{x\in\R^d} f_i(x) = f_i(x^*)$ for each $i$, where $x^* \in \calX^*$ is a minimizer of $f$.
\end{definition}

This condition is strong. Nevertheless, there are specific problems (often overparameterized) satisfying this condition, such as overdetermined robust linear regression, finding a point in the intersection of convex sets, data interpolation in machine learning, etc. We refer the reader to \cite[Section 4]{asi2019stochastic} for a detailed discussion. 

Let us consider the stochastic subgradient method (hereafter, $\SSM$) for solving problem \eqref{eq:problem} with $f$ having the finite-sum structure \eqref{eq:finite sum}. $\SSM$ performs the update
\be\label{eq:SSM}
	\left[
	\begin{split}
		&\text{Sample} \ i_k \ \text{from} \ \{1,\ldots, n\} \ \text{uniformly at random},\\
		&x^{k+1} = x^k - \alpha_{k} g_{i_k}(x^k)
	\end{split}
	\right.
\ee
in the $k$-th iteration, where $g_{i_k}(x^k)\in \partial f_{i_k}(x^k)$ is an unbiased stochastic approximation of the full subgradient, i.e., $\Exp_{i_k}[g_{i_k}(x^k)] \in \partial f(x^k)$. By standard assumption in stochastic optimization, we assume that the realizations $i_0, i_1,\ldots i_k, \ldots$ are independent of each other \cite{nemirovski2009robust}. We then have the following complexity result for $\SSM$.
\begin{corollary}\label{cor:stochastic subgrad}
	Suppose that $f$ in \eqref{eq:problem} has the finite-sum structure \eqref{eq:finite sum} with each $f_i$ being convex.  Suppose further that $f$ satisfies the interpolation condition (see \Cref{def:easy}) and the step sizes  $\{\alpha_k\}_{k= 0}^T$ of $\SSM$ \eqref{eq:SSM} satisfy 
	\be\label{eq:step sizes stoc}
	\alpha_{k}=\frac{\beta_k}{\|g_{i_k}(x^k)\|}\quad  \text{with} \quad \beta_k= \frac{c}{\sqrt{T+1}} \quad \text{for} \ k= 0,1,\ldots, T,
	\ee
	where $c>0$ is some constant.  Let $\tilde{x}^T=\frac{1}{T+1}\sum_{k=0}^T x^k$ and $x^*\in \calX^*$ be some fixed optimal solution to problem \eqref{eq:problem}. Then, we have
	\[
	    \Exp[f(\tilde{x}^T)-f^*] \leq \frac{\sL_{\calA'}\left(\|x^{0}-x^*\|^2 /c + c\right)}{2\sqrt{T+1}},
	\]
	where $\sL_{\calA'}$ is the Lipschitz constant of $f$ on $\calA'$ defined in \Cref{theo:complexity_cvx}.
\end{corollary}

\begin{proof}
	From the update of $\SSM$, we can compute 
	\[
	   \|x^{k+1} - x^*\|^2 = \|x^k - x^*\|^2 - 2\alpha_k \langle g_{i_k}(x^k), x^k - x^* \rangle + \beta_k^2. 
	\]
	Utilizing the convexity of $f_i$ gives 
    \be\label{eq:stoc recur 1}
	   \|x^{k+1} - x^*\|^2 = \|x^k - x^*\|^2 - 2\alpha_k \left(f_{i_k}(x^k) - f_{i_k}(x^*)\right) + \beta_k^2. 
	\ee
	Note that $f_{i_k}(x^k) - f_{i_k}(x^*) \geq 0$ since $f$ satisfies the interpolation condition (see \Cref{def:easy}). It follows that
	\[
	   \|x^{k+1} - x^*\|^2 \leq  \|x^k - x^*\|^2 + \beta_k^2. 
	\]
	Unrolling this inequality up to any $k\leq T$ and invoking the definition of $\beta_k$ in \eqref{eq:step sizes stoc} gives 
	\[
	   \|x^{k+1}-x^*\|^2\leq\|x^{0}-x^*\|^2+\sum_{j=0}^k\beta_j^2\leq\|x^{0}-x^*\|^2+c^2,\quad  0  \leq k\leq T.
	\]
	This shows that the trajectory $\{x^k\}_{k=0}^T$ lies in the ball $ \calA := \{x\in\R^d: \|x-x^*\|^2 \leq \|x^{0}-x^*\|^2+c^2\}$. Next, by following the derivation of \eqref{eq:complexity bound subgrad cvx}, we have $\|g_{i_k}(x^k)\|\leq \sL_{\calA'}$ for all $0\leq k\leq T$,  where $\sL_{\calA'}>0$ is the Lipschitz constant of $f$ on $\calA'$. Invoking this  upper bound in \eqref{eq:stoc recur 1} and using the nonnegativity of $f_{i_k}(x^k) - f_{i_k}(x^*)$ yields
	\[
	   \|x^{k+1} - x^*\|^2 = \|x^k - x^*\|^2 - 2\frac{\beta_k}{\sL_{\calA'}} \left(f_{i_k}(x^k) - f_{i_k}(x^*)\right)  + \beta_k^2. 
	\]
	Upon taking expectation with respect to $i_k$ and then taking total expectation, we obtain 
	\[
	   \Exp\left[\|x^{k+1} - x^*\|^2\right] = \Exp \left[\|x^k - x^*\|^2\right] - 2\frac{\beta_k}{\sL_{\calA'}} \Exp[f(x^k) - f^*]+ \beta_k^2. 
	\]
	Finally, by following the argument after \eqref{eq:recursion subgrad cvx}, we obtain the desired result. 
\end{proof}

It is worth noting that with similar techniques, we can derive an iteration complexity result for the incremental subgradient method.

\subsection{Proximal Subgradient Method} \label{sec:Prox_SubGrad}
In this subsection, we consider the regularized nonsmooth optimization problem 
\begin{equation}\label{eq:problem_prox}
	\min_{x\in \R^d} \ \varphi(x)=f(x)+r(x),
\end{equation}
where $f:\R^d\rightarrow\R$ is convex / weakly convex and $r:\R^d\rightarrow (-\infty, \infty]$ is proper, lower-semicontinuous, and simple in the sense that its proximal operator \eqref{eq:PM} admits a closed-form solution. Since $f$ may be nonsmooth, we utilize a subgradient oracle for $f$.  Consider applying a proximal / projected subgradient method (henceforth, $\ProxSubGrad$) to tackle problem \eqref{eq:problem_prox}, i.e., 
\be\label{eq:proxsubgrad}
x^{k+1} = \prox_{\alpha_k,r}(x^{k} - \alpha_{k}g(x^k)).
\ee

The main assumption we need in this subsection is that $r$ is Lipschitz continuous with parameter $\sL_r>0$ on its effective domain, i.e., $r(x) - r(y) \leq \sL_r \|x - y\|$ for all $x,y\in \operatorname{dom} r$. Such an assumption holds for many different choices of $r$, such as norm functions or indicator functions associated with closed convex sets (which are convex), or the sparsity-inducing functions SCAD \cite{SCAD} and MCP  \cite{MCP} (which are weakly convex).

%\begin{algorithm}[t]
%	\caption{$\ProxSubGrad$: Proximal subgradient method for solving \eqref {eq:problem_prox}}
%	{\bf Input:}  Initial point $x^0$;
%	\begin{algorithmic}[1]
%		\For{$k=0,1,\ldots$}
%		\State Compute a subgradient $g(x^k)\in\partial f(x^k)$;
%		\State Update the step size $\alpha_{k}$ according to a certain rule;
%		\State Update  $x^{k+1} = \prox_{\alpha_k,r}(x^{k} - \alpha_{k}g(x^k))$.
%		\EndFor
%	\end{algorithmic}
%	\label{alg:Prox-SubGrad}
%\end{algorithm}

In the following lemma, we derive important recursions for $\ProxSubGrad$ in both the convex and weakly convex cases. 
\begin{lemma}\label{lemma:recursion ProxSubGrad}
	Consider problem \eqref{eq:problem_prox}, where $r$ is assumed to be Lipschitz continuous with parameter $\sL_r$ on $\operatorname{dom} r$. Suppose that the step sizes $\{ \alpha_k \}$ of $\ProxSubGrad$ \eqref{eq:proxsubgrad} satisfy  $\alpha_k={\beta_k}/{(\|g(x^k)\|+\sL_{r})}$ with $\beta_k>0$ for all $k\geq 0$. Then, for any $x\in \R^d$, the following recursions hold:
	\begin{enumerate}[label=\textup{\textrm{(\alph*)}},topsep=0pt,itemsep=0ex,partopsep=0ex]
		\item If $f$ is convex, then 
		\be\label{eq:pre recur cvx}
		\|x^{k+1} - x\|^2 \leq  \|x^k - x\|^2 - 2\alpha_k  (\varphi(x^k) - \varphi(x)) + \beta_k^2.
		\ee
		\item If $f$ is $\rho$-weakly convex, $r$ is $\eta$-weakly convex, and $\alpha_k<\frac{1}{\eta}$ for all $k\geq 0$, then
	  \be\label{eq:pre recur wcvx}
		 (1-\alpha_k\eta)\|x^{k+1}-x\|^2 \leq \|x^k-x\|^2 - 2\alpha_k\left( \varphi(x^k) - \varphi(x) - \frac{\rho}{2} \|x^k-x\|^2 \right) + \beta_k^2.
		\ee
	\end{enumerate}
\end{lemma}

\begin{proof}
	It suffices to prove the result in (b), as that in (a) is its special case by setting $\rho = \eta = 0$. By the subproblem of $\ProxSubGrad$ \eqref{eq:proxsubgrad}, we obtain 
	\[
	\begin{aligned}
		&\langle g(x^k), x-x^k \rangle + \frac{1}{2\alpha_k} \|x - x^k\|^2 + r(x) \\
		&\geq \langle g(x^k), x^{k+1}-x^k \rangle + \frac{1}{2\alpha_k} \|x^{k+1} - x^k\|^2 + r(x^{k+1}) +  \frac{({1}/{\alpha_k} -\eta)}{2}\|x^{k+1} - x\|^2, \quad \forall x\in \R^d.
	\end{aligned}
	\]
	Multiplying both sides by $2\alpha_k$ and rearranging, we can compute 
	\begin{align*}
		(1-\alpha_k\eta)\|x^{k+1}-x\|^2&\leq\|x^k-x\|^2 - 2\alpha_k\langle g(x^k), x^k-x \rangle - 2\alpha_k (r(x^{k}) - r(x))- \|x^{k+1}-x^{k}\|^2  \\
		&\quad  +2\alpha_k (r(x^k) - r(x^{k+1})) +  2\alpha_k\langle g(x^k), x^k-x^{k+1} \rangle   \\
		&\leq \|x^k-x\|^2 - 2\alpha_k\left( \varphi(x^k) - \varphi(x) - \frac{\rho}{2} \|x^k-x\|^2 \right) - \|x^{k+1}-x^{k}\|^2   \\
		&\quad  + 2\alpha_k(\sL_r + \|g(x^k)\|) \|x^{k+1} - x^k\|  \\
		&\leq \|x^k-x\|^2 - 2\alpha_k\left( \varphi(x^k) - \varphi(x) - \frac{\rho}{2} \|x^k-x\|^2 \right) + \beta_k^2, 
	\end{align*}
	where we have used the $\rho$-weak convexity of $f$, the $\sL_r$-Lipschitz continuity of $\varphi$, and the Cauchy-Schwarz inequality in the second inequality, while the last inequality follows from the definition of $\alpha_k$ and Young's inequality. 
\end{proof}

Equipped with \Cref{lemma:recursion ProxSubGrad}, we are ready to provide complexity results for $\ProxSubGrad$. 
\begin{corollary}\label{cor:compelxity ProxSubGrad}
	Consider problem \eqref{eq:problem_prox}, where $r$ is assumed to be Lipschitz continuous with parameter $\sL_r$ on $\operatorname{dom} r$. Suppose that the step sizes  $\{\alpha_k\}_{k= 0}^T$ of $\ProxSubGrad$ \eqref{eq:proxsubgrad} satisfy
	\[
	   \alpha_k=\frac{\beta_k}{\|g(x^k)\|+\sL_{r}}\quad  \text{with} \quad \beta_k= \frac{c}{\sqrt{T+1}} \quad \text{for} \ k= 0,1,\ldots, T,
	\]
	where $c>0$ is some constant.  Then, the following complexity results hold:
	\begin{enumerate}[label=\textup{\textrm{(\alph*)}},topsep=0pt,itemsep=0ex,partopsep=0ex]
		\item Suppose further that $f$ is convex. 	Let $\tilde{x}^T=\frac{1}{T+1}\sum_{k=0}^T x^k$ and $x^*\in \calX^*$ be some fixed optimal solution to problem \eqref{eq:problem_prox}. Then, we have
		\[
		 	\varphi(\tilde{x}^{T})-\varphi^*\leq \frac{(\sL_r+\sL_{\mathcal A'})\left(\|x^0-x^*\|^2/c + c\right)}{2\sqrt{T+1}},
		\]
		where ${\sf L}_{\mathcal A'}$ is the Lipschitz constant of $f$ on $\calA'$ defined in \Cref{theo:complexity_cvx}.
		
		\item Suppose further that $f$ is $\rho$-weakly convex, $r$ is $\eta$-weakly convex, $\varphi$ has bounded sublevel sets, and  $0<\beta_k\leq \frac{\sL_{r}}{2\eta}$  for $ k= 0,1,\ldots, T$.   Let $\varphi_{\lambda}$ be the Moreau envelope of $\varphi$ with $\lambda<{1}/{(\rho+\eta)}$. Then, the trajectory $\{x^k\}_{k=0}^T$ lies in the bounded sublevel set  $ \calH:= \{x\in \R^d: \varphi_\lambda (x) \leq \varphi_{\lambda}(x^0) + c^2/\lambda\}$. Consequently, for any bounded open convex set $\calH'$ that contains $\calH$, we have 
    	\[
	    \min_{0\leq k\leq T}\|\nabla \varphi_{\lambda}(x^k)\|^2 \leq \frac{(\sL_r + \sL_{\calH'})\left(\varphi_{\lambda}(x^0)-\varphi^*+ c^2 / \lambda \right)}{c(1-\lambda(\rho+\eta))\sqrt{T+1}},
    	\]
	where $\sL_{\calH'}$ is the Lipschitz constant of $f$ on $\calH'$.
	\end{enumerate}
\end{corollary}

\begin{proof}
	We first prove the result in (a). Invoking $x = x^*$ and the definition of $\alpha_k$ in \eqref{eq:pre recur cvx} gives  
	\be\label{eq:cvx_prox_recursion}
	\|x^{k+1} - x^*\|^2 \leq  \|x^k - x\|^2 - \frac{2\beta_k}{\sL_r + \|g(x^k)\|} (\varphi(x^k) - \varphi^*) + \beta_k^2.
	\ee
	The rest of the proof follows that of \Cref{theo:complexity_cvx} since \eqref{eq:cvx_prox_recursion} has almost the same structure as \eqref{eq:complexity_cvx_1}. In particular, the recursion \eqref{eq:cvx_prox_recursion} allows us to establish $\{x^k\}_{k=0}^T \in \calA$ and $\|g(x^k)\|\leq \sL_{\calA'}$ for all $0\leq k \leq T$, where $\calA$ and $\sL_{\calA'}$ are the ball and the Lipschitz constant of $f$ on $\calA'$ defined in \Cref{theo:complexity_cvx}, respectively. Finally, plugging this upper bound on  $\|g(x^k)\|$ into \eqref{eq:cvx_prox_recursion} and following the arguments in the proof of \Cref{theo:complexity_cvx} yield the result.
	
	We next establish the result in (b).  	Denote $\widehat{x}^k:=\prox_{\lambda,\varphi}(x^k)$. Plugging $x = \widehat{x}^k$ in \eqref{eq:pre recur wcvx} gives
	\begin{equation}\label{eq:wcvx_prox_recr2}
		\begin{aligned}
			(1-\alpha_k\eta)&\|x^{k+1}-\widehat{x}^k\|^2\leq\|x^{k}-\widehat{x}^k\|^2-2\alpha_k\left[\varphi(x^{k})-\varphi(\widehat{x}^k)-\frac{\rho}{2}\|x^k - \widehat{x}^k\|^2\right]+\beta_k^2\\
			&\leq(1-\alpha_k\eta)\|\widehat{x}^k-x^{k}\|^2-2\alpha_k\left(\varphi(x^{k})-\varphi_\lambda(\widehat{x}^k)+\frac{1}{2}(\frac{1}{\lambda}-\rho-\eta)\|\widehat{x}^k-x^k\|^2\right)+\beta_k^2.
		\end{aligned}
	\end{equation}
	
	The definitions of Moreau envelope~\eqref{eq:ME} and proximal mapping~\eqref{eq:PM} imply that
	\be\label{eq:wcvx_prox_recr3}
	\varphi_{\lambda}(x^{k+1})\leq \varphi(\widehat{x}^k)+\frac{1}{2\lambda}\|\widehat{x}^k-x^{k+1}\|^2.
	\ee
	Combining \eqref{eq:wcvx_prox_recr2}, \eqref{eq:wcvx_prox_recr3}, and \eqref{eq:envelope properties} yields
	\be\label{eq:wcvx_prox_recr4}
	\begin{aligned}
		\varphi_{\lambda}(x^{k+1})&\leq\varphi_{\lambda}(x^k)-\frac{\alpha_k(1-\lambda(\rho+\eta))}{(1-\alpha_k\eta)}\|\nabla\varphi_{\lambda}(x^k)\|^2+\frac{\beta_k^2}{2\lambda(1-\alpha_k\eta)} \\
		&\leq\varphi_{\lambda}(x^k)- \alpha_k(1-\lambda(\rho+\eta))\|\nabla\varphi_{\lambda}(x^k)\|^2+\frac{\beta_k^2}{\lambda}\\
		&= \varphi_{\lambda}(x^k)- \frac{\beta_k(1-\lambda(\rho+\eta))}{\|g(x^k)\| + \sL_r}\|\nabla\varphi_{\lambda}(x^k)\|^2+\frac{\beta_k^2}{\lambda},
	\end{aligned}
	\ee
	where we used $0\leq \alpha_k \eta \leq \frac{1}{2}$ in the second inequality.  The rest of the proof follows that of \Cref{theo:complexity wcvx} since \eqref{eq:wcvx_prox_recr4} has almost the same structure as \eqref{eq:complexity bound2-1_wcvx}. Specifically, we conclude that the trajectory $\{x^k\}_{k=0}^T$ lies in the bounded sublevel set $ \calH := \{x\in \R^d: \varphi_\lambda (x) \leq \varphi_{\lambda}(x^0) + c^2/\lambda\}$. Consequently, for any bounded open convex set $\calH'$ that contains $\calH$, we can derive $\|g(x^k)\|\leq \sL_{\calH'}$ for all $0\leq k \leq T$,  where $\sL_{\calH'}$ is the Lipschitz constant of $f$ on $\calH'$. Once we plug this upper bound into \eqref{eq:wcvx_prox_recr4} and then follow the arguments in the proof of \Cref{theo:complexity wcvx}, we obtain the desired result.
\end{proof}

\section{Conclusion and Discussion}
In this work, we extended the classic complexity and convergence results for $\SubGrad$ by establishing the $\mathcal{O}(\epsilon^{-2})$ and $\mathcal{O}(\epsilon^{-4})$ iteration complexity bounds for the minimization of non-Lipschitz convex and weakly convex functions, respectively. Convergence results for $\SubGrad$ with proper diminishing step size rules were also established. If $f$ is convex, then we established a $\calO(\log(k)/\sqrt{k})$ rate of convergence in terms of the suboptimality gap. For convex $f$ that possesses the quadratic growth property, an $\calO(1/k)$ convergence rate was derived in terms of squared distance to the optimal solution set.  For weakly convex $f$, we provided asymptotic convergence result.  Moreover, we showed that our analysis techniques can be adapted to derive complexity results for the truncated, stochastic, and proximal subgradient methods in the non-Lipschitz setting.

We believe that our plugin-type analysis can potentially lead to further complexity and convergence results for other variants of the subgradient method in the non-Lipschitz setting. One of the most natural directions would be to investigate if the analysis of the stochastic subgradient method can be extended to  non-interpolating and nonconvex settings. 

\revise{
\section*{Acknowledgments}
We would like to thank the anonymous reviewer for the detailed and constructive comments, which have helped to improve the quality and presentation of the manuscript.
}

\bibliography{paper}
\bibliographystyle{siamplain}

\end{document}